\documentclass[12pt, a4paper, english, reqno]{amsart}
\usepackage{a4wide}
\usepackage[utf8]{inputenc}
\usepackage{amsmath}
\usepackage{amsthm,amssymb}
\usepackage{nicefrac,amsfonts}
\usepackage{thmtools,mathtools,leftindex,leftidx}
\usepackage{caption,subcaption}
\usepackage{epsfig,graphicx,graphics,color,tikz,tikz-cd,tcolorbox}
\usepackage{enumerate,enumitem}
\usepackage[sort,nocompress]{cite}
\usepackage{xspace}
\usepackage{bbm}
\usepackage{hyperref}
\usepackage[capitalise]{cleveref}
\usepackage{float}
\hypersetup{
    colorlinks=true,       
    linkcolor=blue,        
    citecolor=magenta,         
    filecolor=magenta,     
    urlcolor=cyan,         
    linktoc=all
}

\usepackage[parfill]{parskip}

\newcommand{\R}{\mathbb{R}}

\newtheorem{Theorem}{Theorem}
\newtheorem{Lemma}{Lemma}
\newtheorem{Corollary}{Corollary}
\newtheorem{Example}{Example}
\newtheorem{Prop}{Proposition}

\title[ultra log-concavity of matroid intersection]{A note on the ultra log-concavity\\ of matroid intersection}
\date{}

\author[Adam Schweitzer]{Adam Schweitzer}
\thanks{Department of Mathematics, KTH Royal Institute of Technology, Stockholm, Sweden.} 
\subjclass[2020]{05B35, 05A20}
\begin{document}
\maketitle
\begin{abstract}
\noindent In 1971, Mason conjectured that the numbers of independent sets of fixed size in a matroid constitute an ultra log-concave sequence. In 2020, this conjecture was proven by Brändén and Huh and independently by Anari, Liu, Gharan and Vinzant. Recently, this result was extended to $M^\natural$-concave functions. In this note, we make the next step by proving it for $M_2^\natural$-concave functions. This shows the same property for the intersection of any pair of matroids (which itself may not be a matroid). 
 Furthermore, we show that this can not be further extended to the intersections of three matroids by including a counterexample of partition matroids.  
\end{abstract}
\section{Overview of results}
For a matroid $M$ with $n$~elements, let $I_k$ denote the number of independent sets of size~$k$. 
In \cite{mason}, Mason formulated a hierarchy of conjectures, that, in its strongest form, states that the sequence $I_0,I_1,\ldots, I_n$ satisfies \emph{ultra log-concavity}, that is, $$\frac{I_k}{\binom{n}{k}} \frac{I_{k+2}}{\binom{n}{k+2}}\leq \left(\frac{I_{k+1}}{\binom{n}{k+1}}\right)^2,\quad  0\leq k \leq n-2.$$

The reader may find more about this conjecture and the basics of matroid theory in \cite[Conjecture 15.2.2]{oxley2006matroid}.
With the recent advances in the theory of completely log-concave and Lorentzian polynomials, this conjecture was proven in \cite{completelylogconcave} by  Anari, Liu, Gharan and Vinzant and in \cite{lorentzian} by Brändén and Huh, independently.

In \cite[Theorem 1.4]{mnatural} this result was generalized to $M^\natural$-concave functions by defining for any finite set $E$ the numbers $$I_{q,\nu;k}=\sum_{I\in \binom{E}{k}}q^{-\nu(I)}$$ for $0<q\leq 1, k\in \mathbb{Z}$ and an $M^\natural$-concave function $ \nu: 2^E \to \R \cup \{-\infty \}$. The authors show that $$\frac{I_{q,\nu;k}}{\binom{|E|}{k}} \frac{I_{q,\nu;k+2}}{\binom{|E|}{k+2}}\leq \left(\frac{I_{q,\nu;k+1}}{\binom{|E|}{k+1}}\right)^2.$$
Note that if we take the function $\nu$ to be the \emph{matroid valuation} function that associates 0 to the independent sets of a matroid and $-\infty$ to any other set, we recover the original statement of Mason's conjecture.

In this paper, we show the following:
\begin{Theorem}\label{main-thm}
For any real number $0 <q\leq 1$, the pair of $M^\natural$-concave functions
$${\nu_1,\nu_2: 2^E \to \R\cup\{-\infty\},}$$ with $\nu=\nu_1+\nu_2$ satisfies $$\frac{I_{q,\nu;k}}{\binom{|E|}{k}} \frac{I_{q,\nu;k+2}}{\binom{|E|}{k+2}}\leq \left(\frac{I_{q,\nu;k+1}}{\binom{|E|}{k+1}}\right)^2.$$
\end{Theorem}
The sum of two $M^\natural$-concave functions need not be  $M^\natural$-concave itself; such functions are called \emph{$M_2^\natural$-concave}. This concept was originally introduced by Murota in \cite{murota}, where the properties of $M^\natural$-concave functions that extend to this larger class are studied. Thus, we may regard this theorem as a strengthening of the statement proven in \cite{mnatural} to $M_2^\natural$-concave functions.

If we take any two matroids $M_1,M_2$ on the same ground set $E$, and take $\nu_1,\nu_2$ to be the two corresponding matroid valuation functions, we get the following corollary, generalizing Mason's conjecture to the intersection of matroids.

\begin{Corollary}
For matroids $M_1,M_2$ on the same ground set with $n$ elements, let $I_k$ denote the number of common independent sets of size $k$. Then, the following holds $$\frac{I_k}{\binom{n}{k}} \frac{I_{k+2}}{\binom{n}{k+2}}\leq \left(\frac{I_{k+1}}{\binom{n}{k+1}}\right)^2.$$
\end{Corollary}
Example~\ref{ex:3part} below is showing that the statement above does not generalize further, even in the case of three partition matroids (see ).

Additionally, we prove the following theorem.
\begin{Theorem}\label{main-thm-2}
For any pair of  $M^\natural$-concave functions $\nu_1,\nu_2: 2^E \to \R\cup\{-\infty\}$ and real number ${0 <q\leq 1}$ let $$J_{\nu_1,\nu_2, q;k}:=\sum_{\substack{S_1\subseteq S_2\\ |S_2\setminus S_1|=k} } q^{-\nu_1(S_1)-\nu_2(S_2)}.$$
Then $$\frac{J_{\nu_1,\nu_2, q;k}}{\binom{|E|}{k}} \frac{J_{\nu_1,\nu_2, q;k+2}}{\binom{|E|}{k+2}}\leq \left(\frac{J_{\nu_1,\nu_2, q;k+1}}{\binom{|E|}{k+1}}\right)^2.$$
\end{Theorem}
\newpage
Similarly to the case before, if we take $\nu_1,\nu_2$ to be the corresponding matroid valuation functions, we get the following corollary.
\begin{Corollary}
For any pair of matroids $M_1,M_2$ on the same groundset $E$ with $n$ elements let $J_k$ be the number of pairs of independent sets of $I_1\in M_1, I_2\in M_2$ such that $I_1\subseteq I_2$ and $|I_2\setminus I_1|=k$. Then  $$\frac{J_k}{\binom{n}{k}} \frac{J_{k+2}}{\binom{n}{k+2}}\leq \left(\frac{J_{k+1}}{\binom{n}{k+1}}\right)^2.$$
\end{Corollary}

If instead we take $\nu_2(S)$ to be $0$ if $E\setminus S\in M_2$ and $-\infty$ otherwise, and $\nu_1$ as before, we obtain the following corollary. Note that this is indeed an $M^\natural$-concave function, see Lemma~\ref{lemma-inv}.
\begin{Corollary}
For any pair of matroids $M_1,M_2$ on the same groundset with $n$ elements, let $H_k$ denote the number of pairs of disjoint independent sets in $M_1$ and $M_2$ that cover $k$ elements of the groundset. Then  $$\frac{H_k}{\binom{n}{k}} \frac{H_{k+2}}{\binom{n}{k+2}}\leq \left(\frac{H_{k+1}}{\binom{n}{k+1}}\right)^2.$$
\end{Corollary}

\section{Preliminaries}
\subsection{Notation}
We refer to a sequence $a_0,a_1,a_2, \ldots, a_n$ as \emph{log-concave} if it is non-negative, has no internal zeroes, i.e., the indices of the non-zero elements form a convex subset of the set of natural numbers $\mathbb N$, and if for every $i\in [1, n-1]$ we have $$a_{i-1}a_{i+1}\leq a_i^2.$$

We refer to a sequence $a_0,a_1, \ldots, a_n$ as \emph{ultra log-concave} if the sequence $$\frac{a_0}{\binom{n}{0}},\frac{a_1}{\binom{n}{1}},  \ldots, \frac{a_n}{\binom{n}{n}}$$ is log-concave. 

For a set $E$, we denote the set of its subsets by $2^E$, and use $\R^E$ to denote the set of real vectors indexed with the elements of $E$. 
The set of subsets of a set $E$ of size~$k$ is denoted by $\binom{E}{k}$. 

For a set $S\subseteq E$ we refer to the vector that is 1 for every element of $S$ and 0 everywhere else as the \emph{characteristic vector of $S$}. We use subsets $S\subseteq E$ and 
characteristic vectors in $\R^E$ interchangeably.

For the vectors  $x, \alpha\in \R^E$ we use the notation $x^\alpha=\prod_i x_i^{\alpha_i}$.
For a pair of vectors $\alpha,\beta\in \R^E$ we use the notation $\alpha < \beta$ (or $\alpha\leq \beta$) when $\alpha_i < \beta_i$ (or $\alpha_i\leq \beta_i$) for each $i$.
For a vector $\alpha\in \R^E$ we use the notation $$\alpha!=\prod_{i\in E}\alpha_i!.$$ Similarly, for vectors $\alpha\leq \beta$ we use the notation $$\binom{\beta}{\alpha}=\frac{\beta!}{\alpha!(\beta-\alpha)!}=\prod_{i\in E} \binom{\beta_i}{ \alpha_i}.$$

For a function $F$ depending on a vector variable $x\in \R^E$ we denote $\partial_x$ the vector of partial derivative operators $(\partial_i | i\in E )$. Thus, consistent with our notation, for a subset $S\subseteq E$ we have that $\partial_x^S=\prod_{i\in S}\partial_i$.

\subsection{Lorentzianity}
Lorentzian polynomials, introduced in \cite{lorentzian}, have a history of being used to resolve or make progress with a number of long-standing conjectures. For example, Brändén and Huh settled Mason's conjecture in \cite{lorentzian}, Hafner, Mészáros and Vidinas showed the log-concavity of the absolute values of the coefficients of the Alexander polynomial for special alternating links in \cite{alex} and An, Tung and Zhang showed that the supports of the Postnikov-Stanley polynomials are M-convex in \cite{poststan}. This concept is a central tool for our proof.

A homogeneous polynomial $P$ with real coefficients in $n$ variables of degree $d$ is said to be \emph{strictly-Lorentzian} if it meets the following two conditions:
\begin{enumerate}
\item All coefficients are positive,
\item For any $i_1,i_2,\ldots, i_{d-2}\in \{1,\ldots, n\}$ the Hessian of the degree 2 polynomial $$\partial_ {i_1}\partial_ {i_2}\cdots \partial{i_{d-2}}P$$ has one positive and $n-1$ negative eigenvalue.
\end{enumerate}
We refer to a polynomial as \emph{Lorentzian} if it is the limit of strictly-Lorentzian polynomials.
As such, all linear polynomials with non-negative coefficients are Lorentzian.

The following property links Lorentzianity to ultra log-concavity.
\begin{Prop}[{\cite[Example 2.3]{lorentzian}}] \label{prop-ultra-lorentz}
Let P be a Lorentzian polynomial of degree $d$ in two variables $x,y$. Let $a_i$ be the coefficient of $x^iy^{d-i}$. Then $a_0,a_1,\ldots ,a_d$ is an ultra log-concave sequence.
\end{Prop}

Lorentzian polynomials interact well with restrictions and products as presented by the following proposition.
\begin{Prop}\label{prop-restrict}
Let $P,Q$ be a Lorentzian polynomial in variables $x_1,x_2,\ldots, x_n$, then,
\begin{enumerate}
\item $P\cdot Q$ is a Lorentzian polynomial \cite[Theorem 2.30 and Corollary 2.32]{lorentzian}.
\item if we restrict any $x_i$ to 0, we obtain either 0 or a Lorentzian polynomial \cite[Lemma 2.20]{lorentzian},
\item if we substitute $x_i=x_j$ we obtain a Lorentzian polynomial \cite[Proposition 2.20]{lorentzian},
\end{enumerate}
\end{Prop}

For a polynomial $P\in \R[x_1,\ldots,x_n]$ we may define the operator $P(\partial_1,\ldots,\partial_n)$ that acts on polynomials in $\R[x_1,\ldots,x_n]$ where each monomial acts by differentiation linearly. Ross, Sü{\ss} and Wannerer in \cite{dually} characterize which of these operators preserve the Lorentzian property.

\begin{Theorem}\label{thm-dual}
The differential operator $P(\partial_1,\ldots,\partial_n)$ preserves the Lorentzian property if and only if the polynomial $N(x^\kappa P(x_1^{-1},\ldots, x_n^{-1}))$ is Lorentzian, where $\kappa$ is any vector such that $\deg_i(P)\leq \kappa_i$ for all $i$ and $N$ is the linear operator that takes $x^\alpha \mapsto \frac{x^\alpha}{\alpha!}$. 
\end{Theorem}

\subsubsection{ $M^\natural$-concavity}
Now let us introduce $M^\natural$-concavity and the recent developments in its connection to Lorentzian polynomials that we utilize for our proof.

We refer to a function $\nu: 2^E \to \R\cup \{-\infty\}$ as \emph{$M^\natural$-concave} if it satisfies the following property. For any $I_1,I_2\subseteq E$ and $i_1 \in I_1\setminus I_2$ we have either
\begin{enumerate}
\item $\nu(I_1) + \nu(I_2) \leq \nu(I_1\setminus \{i_1\}) + \nu(I_2 \cup \{i_1\})$ or
\item there is $i_2 \in I_2\setminus I_1$ such that $\nu(I_1) + \nu(I_2) \leq \nu(I_1\setminus \{i_1\} \cup \{i_2\}) + \nu(I_2 \setminus \{i_2\} \cup \{i_1\}).$
\end{enumerate}
This concept was introduced by Murota and Shioura in \cite{mnaturaldef} as a discrete analogue of convexity. For a more complete introduction to this subject, we refer the reader to the book \cite{murota} by Murota.

An important class of examples are the \emph{matroid valuation functions}. If we take a matroid $M$ on the ground set $E$ the function $\nu_M: 2^E\to \R\cup\{-\infty\}$, that acts as $$S \mapsto \begin{cases}0 \text{ if $S$ is an independent set in $M$} \\ -\infty \text{ otherwise} \end{cases},$$ is $M^\natural$-concave \cite[Example 1.1]{mnatural}.

In \cite{mnatural} the authors show the following theorem:
\begin{Theorem} \label{thm-charact}
For a function $\nu: 2^E\to \R \cup\{-\infty\}$ and positive real number $q$ for any $x\in \R^E, y\in \R$ let $$Z_{\nu,q}(x,y):=\sum_{S\subseteq E} q^{-\nu(S)}x^S y^{|E\setminus S|}.$$
Then $\nu$ is  $M^\natural$-concave if and only if $Z_{\nu,q}$ is Lorentzian for every $q\leq 1$.
\end{Theorem}

\section{Ultra log-concavity in $M_2^\natural$-concave functions}
Let us first show that our main results cannot hold for the intersection of three matroids.
\begin{Example}\label{ex:3part}

Let us show that there exist three partition matroids such that the numbers of size $k$ common independent sets do not constitute a log-concave sequence.

Let our groundset be the set $$E=\{a_1,a_2,a_3,b_1,b_2,\ldots,b_m\}$$ for some number $m$. Now, let $M_i$ be the partition matroid with one part being $$\{a_i,b_1,b_2,\ldots,b_m\}$$ and each other $a_j$ constitutes a part on its own.

The intersection contains every subset of $\{a_1,a_2,a_3\}$  as no pair of them is ever contained in the same part of any $M_i$. On the other hand, if $b_j$ is in a common independent set $I$, no $a_i$ can be in it, as $a_i$ and $b_j$ are in the same part of $M_i$. Thus, the intersection is exactly the subsets of $\{a_1,a_2,a_3\}$ and the sets $\{b_i\}$.

Thus, we have only one 3-element subset, three 2-element subsets, but $m+3$ different 1-element subsets. If we choose $m$ to be at least seven, we have that $(m+3)\cdot1> 3^2$, thus log-concavity fails, thus so does ultra log-concavity.

\end{Example}

Let us introduce a few auxiliary lemmas for the proof of our main theorem.

\begin{Lemma}\label{lemma-inv}
For any $M^\natural$-concave function   $\nu: 2^E\to \R \cup \{-\infty\}$ the function defined as $\nu^*(S):=\nu(E\setminus S)$ is also  $M^\natural$-concave.
\end{Lemma}
\begin{proof}
We check the definition directly. Let us take $I_1,I_2\subset E$ and $i_1\in I_1\setminus I_2$, let us show that one of the two conditions must hold. Let us apply the definition of  \mbox{$M^\natural$-concavity} of $\nu$ to the sets $E\setminus I_2$ and $E\setminus I_1$ and the element $i_1 \in (E\setminus I_2) \setminus (E\setminus I_1)$. Now, either $$\nu(E\setminus I_2) + \nu(E \setminus I_1) \leq \nu(E\setminus I_2 \setminus \{i_1\}) + \nu(E \setminus I_1 \cup \{i_1\}),$$ applying the definition of $\nu^*$ $$\nu^*(I_2)+\nu^*(I_1)\leq \nu^*(I_2\cup\{i_1\})+\nu^*(I_1\setminus \{i_1\}),$$ thus the first condition holds for $\nu^*$, or there exists $i_2\in (E \setminus I_1)\setminus  (E \setminus I_2)=I_2\setminus I_1$ such that
$$\nu(E \setminus I_2) + \nu(E \setminus I_1) \leq \nu((E \setminus I_2)\setminus \{i_1\} \cup  \{i_2\}) + \nu((E \setminus I_1) \setminus \{i_2\} \cup \{i_1\})$$ applying the definition of $\nu^*$ $$\nu^*(I_2)+\nu^*(I_1) \leq \nu^*(I_2\cup\{i_1\}\setminus\{i_2\})+\nu^*(I_1\setminus\{i_1\}\cup \{i_2\}),$$ thus the second condition holds for $\nu^*$. Thus, $\nu^*$ is indeed  $M^\natural$-concave.
\end{proof}

\begin{Lemma} \label{lemma-operator}
For $0<q\leq 1$ and $M^\natural$-concave function $\nu: 2^E\to \R \cup \{-\infty\}$, if we define for $x\in \R^E,y\in \R$$$P_{\nu,q}(x,y)= \sum_{S\subseteq E}  q^{-\nu(S)}x^S y^{|E\setminus S|} \frac{|S|!}{|E|!},$$
then, the differential operator $P_{\nu,q}(\partial_{x},\partial_y)$ preserves the Lorentzian property.
\end{Lemma}
\begin{proof}
We check the condition of Theorem~\ref{thm-dual}. Choose $\kappa$ to be 1 on the variables $x_i$ and $|E|$ for the variable $y$.
Thus, it is sufficient to prove that the polynomial $$\frac{1}{|E|!}\sum_{S\subseteq E}  q^{-\nu(S)}x^{E\setminus S} y^{|S|}$$ is Lorentzian.

Let us take the function $\nu^*(S):=\nu(E\setminus S)$. This is  $M^\natural$-concave by Lemma~\ref{lemma-inv}. Theorem~\ref{thm-charact} applied to $\nu^*$ and the fact that Lorentzianity is unchanged under products with positive constants shows that this polynomial is Lorentzian.
\end{proof}

Now we are ready to prove our main result that implies our other claims immediately.

\begin{Theorem}
For $M^\natural$-concave functions  $\nu_1,\nu_2:2^E \to \R\cup \{-\infty\}$ and real number $0< q \leq  1$ the polynomial $$R_{\nu_1,\nu_2,q}:=\sum_{S_2 \subseteq S_1\subseteq E}  q^{-\nu_1(S_1)-\nu_2(S_2)}x^{S_1\setminus S_2} y^{|S_2|}z^{|E \setminus S_1|}$$ is Lorentzian.
\end{Theorem}
\begin{proof}
We show that we may get $R_{\nu_1,\nu_2,q}$ as follows: take the function $y^{|E|}Z_{\nu_1,q}(x,z)$ and apply the operator $P_{\nu,q}(\partial_{x},\partial_y)$ to it.

Let us determine the resulting polynomial. By the linearity of differentiation, we get $$\sum_{S_1\subseteq E} \sum_{S_2\subseteq E}  q^{-\nu_1(S_1)}x^{S_1} y^{|E|} z^{|E \setminus S_1|} q^{-\nu_2(S_2)} \partial_{x}^{S_2} \partial_{y}^{|E\setminus S_2|}\frac{|S_2|!}{|E|!}.$$
By evaluating the differentiation, we get:
$$\sum_{S_2 \subseteq S_1\subseteq E}  q^{-\nu_1(S_1)-\nu_2(S_2)}x^{S_1\setminus S_2} \frac{|E|!}{|S_2|!}y^{|S_2|}z^{|E \setminus S_1|}\frac{|S_2|!}{|E|!}$$
After simplifying, this is exactly the definition of $R_{\nu_1,\nu_2,q}$. As we have obtained $R_{\nu_1,\nu_2,q}$ by applying an operator preserving the Lorentzian property by Lemma~\ref{lemma-operator} to a Lorentzian polynomial by Theorem~\ref{thm-charact}, the statement follows.
\end{proof}

Let us now derive our other theorems from this claim.

\begin{proof}[Proof of Theorem \ref{main-thm}]
Let us restrict $R_{\nu_1,\nu_2,q}$ to $x=0$, resulting in
$$\sum_{ S_1=S_2\subseteq E}  q^{-\nu_1(S_1)-\nu_2(S_2)}y^{|S_2|}z^{|E \setminus S_1|}=\sum_{ S_1=S_2\subseteq E}  q^{-\nu(S)}y^{|S|}z^{|E \setminus S|}.$$ This is either Lorentzian or 0 by Proposition~\ref{prop-restrict}. As the coefficients of this $y^kz^{|E|-k}$ are exactly $I_{\nu,q;k}$, the ultra log-concavity follows immediately by Proposition~\ref{prop-ultra-lorentz}.
\end{proof}

\begin{proof}[Proof of Theorem \ref{main-thm-2}]
Let us restrict $R_{\nu_1,\nu_2,q}$ to $x_i=x_j$ (for all $i,j\in E$) and $y=z$, resulting in
$$\sum_{ S_1\subseteq S_2\subseteq E}  q^{-\nu_1(S_1)-\nu_2(S_2)}x^{|S_1\setminus S_2|} z^{|E \setminus (S_1\setminus S_2)|}.$$ As the coefficients of this $x^kz^{|E|-k}$ are exactly $J_{\nu_1,\nu_2,q;k}$, the ultra log-concavity follows immediately by Proposition~\ref{prop-ultra-lorentz}.
\end{proof}
\section*{Acknowledgements}
The author thanks Georg Loho posing this question. The author thanks Georg Loho, Kristóf Bérczi, Péter Csikvári, Benjamin Schröter and Petter Brändén for stimulating discussions and helpful assistance.
The author was supported by the Knut and Alice Wallenberg Foundation. 
\section*{Statement on AI Usage}
AI tools were not used in this work.
\bibliographystyle{alpha}

\bibliography{bibliography}

\end{document}